\documentclass[a4paper,12pt]{amsart}
\usepackage{graphicx}
\usepackage{amsmath}
\usepackage{amssymb}

\newtheorem{theorem}{Theorem}[section]
\newtheorem*{thmA}{Theorem A}
\newtheorem{lemma}[theorem]{Lemma}
\newtheorem{corollary}[theorem]{Corollary}

\theoremstyle{definition}
\newtheorem{remark}[theorem]{Remark}

\renewcommand{\Im}{\,{\operatorname{Im}\,}}

\renewcommand{\Re}{{\operatorname{Re}\,}}

\newcommand{\inv}{^{-1}}

\renewcommand{\arg}{\,{\operatorname{arg}\,}}

\newcommand{\aand}{{\quad\text{and}\quad}}

\numberwithin{equation}{section}
\usepackage{float}

\begin{document}
\title{Some extensions of the Open Door Lemma}
\author[M. Li]{Ming Li}
\address{Graduate School of Information Sciences, \\
Tohoku University, \\
Aoba-ku, Sendai 980-8579, Japan}
\email{li@ims.is.tohoku.ac.jp}
\author[T. Sugawa]{Toshiyuki Sugawa}
\address{Graduate School of Information Sciences, \\
Tohoku University, \\ 
Aoba-ku, Sendai 980-8579, Japan}
\email{sugawa@math.is.tohoku.ac.jp}

\subjclass{Primary 30C45; Secondary 30C80}
\keywords{open door function, subordination}
\begin{abstract}
Miller and Mocanu proved in their 1997 paper
a greatly useful result which is now known as the Open Door Lemma.
It provides a sufficient condition for an analytic function on the
unit disk to have positive real part.
Kuroki and Owa modified the lemma when the initial point is non-real.
In the present note, by extending their methods, we give
a sufficient condition for an analytic function on the unit disk
to take its values in a given sector.
\end{abstract}
\thanks{
The authors were supported in part by JSPS Grant-in-Aid for
Scientific Research (B) 22340025.
}
\maketitle

\section{Introduction}
We denote by $\mathcal{H}$ the class of holomorphic functions on the unit disk
$\mathbb{D}=\{z: |z|<1\}$ of the complex plane $\mathbb{C}.$
For $a\in\mathbb{C}$ and $n\in\mathbb{N},$ let $\mathcal{H}[a,n]$ denote the subclass of
$\mathcal{H}$ consisting of functions $h$ of the form $h(z)=a+c_nz^n+c_{n+1}z^{n+1}+\cdots.$
Here, $\mathbb{N}=\{1,2,3,\dots\}.$
Let also $\mathcal{A}_n$ be the set of functions $f$ of the form $f(z)=zh(z)$ for
$h\in\mathcal{H}[1,n].$

A function $f\in\mathcal{A}_1$ is called {\it starlike} (resp.~{\it convex})
if $f$ is univalent on $\mathbb{D}$ and if the image $f(\mathbb{D})$ is starlike with respect
to the origin (resp.~convex).
It is well known (cf.~\cite{Duren:univ}) that $f\in\mathcal{A}_1$ is starlike
precisely if $q_f(z)=zf'(z)/f(z)$ has positive real part on $|z|<1,$
and that $f\in\mathcal{A}_1$ is convex precisely if $\varphi_f(z)=1+zf''(z)/f'(z)$
has positive real part on $|z|<1.$
Note that the following relation holds for those quantities:
$$
\varphi_f(z)=q_f(z)+\frac{zq_f'(z)}{q_f(z)}.
$$
It is geometrically obvious that a convex function is starlike.
This, in turn, means the implication
$$
\Re\left[q(z)+\frac{zq'(z)}{q(z)}\right]>0~\text{on}~|z|<1
\quad\Rightarrow\quad
\Re q(z)>0~\text{on}~|z|<1
$$
for a function $q\in\mathcal{H}[1,1].$
Interestingly, it looks highly nontrivial.
Miller and Mocanu developed a theory (now called {\it differential
subordination}) which enables us to deduce such a result systematically.
See a monograph ~\cite{MM:ds} written by them for details.

The set of functions $q\in\mathcal{H}[1,1]$ with $\Re q>0$ is called the
Carath\'eodory class and will be denoted by $\mathcal P.$
It is well recognized that the function $q_0(z)=(1+z)/(1-z)$
(or its rotation) maps the unit disk univalently onto the right
half-plane and is extremal in many problems.
One can observe that the function
$$
\varphi_0(z)=q_0(z)+\frac{zq_0'(z)}{q_0(z)}
=\frac{1+z}{1-z}+\frac{2z}{1-z^2}
=\frac{1+4z+z^2}{1-z^2}
$$
maps the unit disk onto the slit domain $V(-\sqrt{3},\sqrt{3}),$ where
$$
V(A,B)=\mathbb{C}\setminus \{iy: y\le A~\text{or}~y\ge B\}
$$
for $A,B\in\mathbb R$ with $A<B.$
Note that $V(A,B)$ contains the right half-plane and has the ``window"
$(Ai, Bi)$ in the imaginary axis to the left half-plane.
The Open Door Lemma of Miller and Mocanu asserts for a function
$q\in\mathcal{H}[1,1]$ that,
if $q(z)+zq'(z)/q(z)\in V(-\sqrt3,\sqrt3)$ for $z\in\mathbb{D},$ then $q\in\mathcal P.$
Indeed, Miller and Mocanu ~\cite{MM97BB} (see also ~\cite{MM:ds})
proved it in a more general form.
For a complex number $c$ with $\Re c>0$ and $n\in\mathbb{N},$ we consider the
positive number
$$
C_{n}(c)=\frac{n}{\Re c}\left[|c|\sqrt{\frac{2\Re c}{n}+1}+\Im{c}\right].
$$
In particular, $C_n(c)=\sqrt{n(n+2c)}$ when $c$ is real.
The following is a version of the Open Door Lemma
modified by Kuroki and Owa ~\cite{KO14}.

\begin{thmA}[Open Door Lemma]\label{Thm:ODL}
Let $c$ be a complex number with positive real part and $n$ be an integer
with $n\ge1.$
Suppose that a function $q\in\mathcal{H}[c,n]$ satisfies the condition
$$
q(z)+\frac{zq'(z)}{q(z)}\in V(-C_n(c),C_n(\bar c)),\quad z\in\mathbb{D}.
$$
Then $\Re q>0$ on $\mathbb{D}.$
\end{thmA}

\begin{remark}
In the original statement of the Open Door Lemma in ~\cite{MM97BB},
the slit domain was erroneously described as $V(-C_n(c),C_n(c)).$
Since $C_n(\bar c)<C_n(c)$ when $\Im c>0,$
we see that $V(-C_n(\bar c),C_n(\bar c))\subset V(-C_n(c),C_n(\bar c))
\subset V(-C_n(c),C_n(c))$ for $\Im c\ge0$ and the inclusions are strict
if $\Im c>0.$
As the proof will suggest us,
seemingly the domain $V(-C_n(c),C_n(\bar c))$ is maximal for
the assertion, which means that the original statement in ~\cite{MM97BB}
and the form of the associated open door function
are incorrect for a non-real $c.$
This, however, does not decrease so much the value of the original article
~\cite{MM97BB} by Miller and Mocanu because the Open Door Lemma is mostly
applied when $c$ is real.
We also note that the Open Door Lemma deals with the function
$p=1/q\in\mathcal{H}[1/c,n]$ instead of $q.$
The present form is adopted for convenience of our aim.
\end{remark}

The Open Door Lemma gives a sufficient condition for $q\in\mathcal{H}[c,n]$
to have positive real part.
We extend it so that $|\arg q|<\pi\alpha/2$ for a given $0<\alpha\le 1.$
First we note that the M\"obius transformation
$$
g_c(z)=\frac{c+\bar cz}{1-z}
$$
maps $\mathbb{D}$ onto the right half-plane in such a way that $g_c(0)=c,$
where $c$ is a complex number with $\Re c>0.$
In particular, one can take an analytic branch of $\log g_c$ so that
$|\Im\log g_c|<\pi/2.$
Therefore, the function $q_0=g_c^\alpha=\exp(\alpha\log g_c)$ maps $\mathbb{D}$ 
univalently onto the sector $|\arg w|<\pi\alpha/2$ in such a way that
$q_0(0)=c^\alpha.$
The present note is based mainly on the following result, which will
be deduced from a more general result of Miller and Mocanu (see Section 2).

\begin{theorem}\label{thm:main}
Let $c$ be a complex number with $\Re c>0$
and $\alpha$ be a real number with $0<\alpha\le 1.$
Then the function
$$
R_{\alpha,c,n}(z)=g_c(z)^\alpha+\frac{n\alpha zg_c'(z)}{g_c(z)}
=\left(\frac{c+\bar{c}z}{1-z}\right)^{\alpha}
+\frac{2n\alpha(\Re c)z}{(1-z)(c+\bar{c}z)}
$$
is univalent on $|z|<1.$
If a function $q\in\mathcal{H}[c^\alpha,n]$ satisfies the condition
$$
q(z)+\frac{zq'(z)}{q(z)}\in R_{\alpha,c,n}(\mathbb{D}),\quad z\in\mathbb{D},
$$
then $|\arg q|<\pi\alpha/2$ on $\mathbb{D}.$
\end{theorem}

We remark that the special case when $\alpha=1$ reduces to Theorem \ref{Thm:ODL}
(see the paragraph right after Lemma \ref{lem:2} below.
Also, the case when $c=1$ is already proved by Mocanu ~\cite{Mocanu86A} even
under the weaker assumption that $0<\alpha\le 2$
(see Remark \ref{rem:Mocanu}).
Since the shape of $R_{\alpha,c,n}(\mathbb{D})$ is not very clear, 
we will deduce more concrete
results as corollaries of Theorem \ref{thm:main} in Section 3.
This is our principal aim in the present note.

\section{Preliminaries}

We first recall the notion of subordination.
A function $f\in\mathcal{H}$ is said to be {\it subordinate} to $F\in\mathcal{H}$
if there exists a function $\omega\in\mathcal{H}[0,1]$ such that $|\omega|<1$
on $\mathbb{D}$ and that $f=F\circ\omega.$
We write $f\prec F$ or $f(z)\prec F(z)$ for subordination.
When $F$ is univalent, $f\prec F$ precisely if $f(0)=F(0)$ and if
$f(\mathbb{D})\subset F(\mathbb{D}).$

Miller and Mocanu ~\cite[Theorem 5]{MM97BB} (see also ~\cite[Theorem 3.2h]{MM:ds})
proved the following general result, from which we will deduce Theorem \ref{thm:main}
in the next section.

\begin{lemma}[Miller and Mocanu]\label{lem:MM}
Let $\mu,\nu\in \mathbb{C}$ with $\mu\neq 0$ and $n$ be a positive integer.
Let $q_0\in\mathcal{H}[c,1]$ be univalent and assume that
$\mu q_0(z)+\nu\neq 0$ for $z\in\mathbb{D}$ and $\Re (\mu c+\nu)>0.$
Set $Q(z)=zq_0'(z)/(\mu q_0(z)+\nu),$ and
\begin{equation}\label{eq:h}
h(z)=q_0(z)+nQ(z)=q_0(z)+\frac{nzq_0'(z)}{\mu q_0(z)+\nu}.
\end{equation}
Suppose further that
\begin{enumerate}
\item[(a)]
$\Re[zh'(z)/Q(z)]=\Re[h'(z)(\mu q_0(z)+\nu)/q_0'(z)]>0,$ and
\item[(b)]
either $h$ is convex or $Q$ is starlike.
\end{enumerate}
If $p\in\mathcal{H}[c,n]$ satisfies the subordination relation
\begin{equation}
q(z)+\frac{zq'(z)}{\mu q(z)+\nu}\prec h(z),
\end{equation}
then $q\prec q_0$, and $q_0$ is the best dominant.
An extremal function is given by $q(z)=q_0(z^n).$
\end{lemma}

In the investigation of the generalized open door function $R_{\alpha,c,n},$
we will need to study the positive solution to the equation
\begin{equation}\label{eq:eq}
x^2+Ax^{1+\alpha}-1=0,
\end{equation}
where $A>0$ and $0<\alpha\le1$ are constants.
Let $F(x)=x^2+Ax^{1+\alpha}-1.$
Then $F(x)$ is increasing in $x>0$ and $F(0)=-1<0,~F(+\infty)=+\infty.$
Therefore, there is a unique positive solution $x=\xi(A,\alpha)$ to the equation.
We have the following estimates for the solution.

\begin{lemma}\label{lem:sol}
Let $0<\alpha\le1$ and $A>0.$
The positive solution $x=\xi(A,\alpha)$ to equation \eqref{eq:eq} satisfies
the inequalities
$$
(1+A)^{-1/(1+\alpha)}\le\xi(A,\alpha)\le (1+A)^{-1/2}~(<1).
$$
Here, both inequalities are strict when $0<\alpha<1.$
\end{lemma}

\begin{proof}
Set $\xi=\xi(A,\alpha).$
Since the above $F(x)$ is increasing in $x>0,$
the inequalities $F(x_1)\le 0=F(\xi)\le F(x_2)$
imply $x_1\le \xi\le x_2$ for positive numbers $x_1, x_2$
and the inequalities are strict when $x_1<\xi<x_2.$
Keeping this in mind, we now show the assertion.
First we put $x_2=(1+A)^{-1/2}$ and observe
$$
F(x_2)=\frac1{1+A}+\frac{A}{(1+A)^{(1+\alpha)/2}}-1
\ge\frac1{1+A}+\frac{A}{1+A}-1=0,
$$
which implies the right-hand inequality in the assertion.

Next put $x_1=(1+A)^{-1/(1+\alpha)}.$
Then
$$
F(x_1)=\frac1{(1+A)^{2/(1+\alpha)}}+\frac{A}{1+A}-1
\le\frac1{1+A}+\frac{A}{1+A}-1=0,
$$
which implies the left-hand inequality.
We note also that $F(x_1)<0<F(x_2)$ when $\alpha<1.$
The proof is now complete.
\end{proof}

\section{Proof and corollaries}

Theorem \ref{thm:main} can be rephrased in the following.

\begin{theorem}
Let $c$ be a complex number with $\Re c>0$
and $\alpha$ be a real number with $0<\alpha\le 1.$
Then the function
$$
R_{\alpha,c,n}(z)=g_c(z)^\alpha+\frac{n\alpha zg_c'(z)}{g_c(z)}
$$
is univalent on $|z|<1.$
If a function $q\in\mathcal{H}[c^\alpha,n]$ satisfies the subordination condition
$$
q(z)+\frac{zq'(z)}{q(z)}\prec R_{\alpha,c,n}(z)
$$
on $\mathbb{D},$ then $q(z)\prec g_c(z)^\alpha$ on $\mathbb{D}.$
The function $g_c^\alpha$ is the best dominant.
\end{theorem}

\begin{proof}
We first show that the function $Q(z)=\alpha zg_c'(z)/g_c(z)$ is starlike.
Indeed, we compute
\begin{equation*}
\frac{zQ'(z)}{Q(z)}=1-\frac{\bar{c}z}{c+\bar{c}z}+\frac{z}{1-z}
=\frac12\left[\frac{c-\bar cz}{c+\bar cz}+\frac{1+z}{1-z}\right].
\end{equation*}
Thus we can see that $\Re[zQ'(z)/Q(z)]>0$ on $|z|<1.$
Next we check condition (a) in Lemma \ref{lem:MM} for the functions
$q_0=g_c^\alpha, h=R_{\alpha,c,n}$ with the choice $\mu=1,\nu=0.$
We have the expression
$$
\frac{zh'(z)}{Q(z)}=q_c(z)^\alpha+n\frac{zQ'(z)}{Q(z)}.
$$
Since both terms in the right-hand side have positive real part,
we obtain (a).
We now apply Lemma \ref{lem:MM} to obtain the required assertion up to
univalence of $h=R_{\alpha,c,n}.$
In order to show the univalence, we have only to note that the condition (a)
implies that $h$ is close-to-convex, since $Q$ is starlike.
As is well known, a close-to-convex function is univalent (see \cite{Duren:univ}),
the proof has been finished.
\end{proof}

We now investigate the shape of the image domain $R_{\alpha,c,n}(\mathbb{D})$
of the generalized open door function $R_{\alpha,c,n}$ given in
Theorem \ref{thm:main}.
Let $z=e^{i\theta}$ and $c=r e^{it}$ for $\theta\in\mathbb R, r>0$ and $-\pi/2<t<\pi/2.$
Then we have
\begin{equation*}
\begin{aligned}
R_{\alpha,c,n}(e^{i\theta})
&=\left(\frac{re^{it}+re^{-it}e^{i\theta}}{1-e^{i\theta}}\right)^{\alpha}
+\frac{2n\alpha e^{i\theta}\cos t}{(1-e^{i\theta})(e^{it}+e^{-it}e^{i\theta})}\\
&=\left(\frac{r\cos{(t-\theta/2)}}{\sin{(\theta/2)}}i\right)^{\alpha}
+\frac i2\cdot
\frac{n\alpha\cos{t}}{\sin{(\theta/2)\cos{(t-\theta/2)}}}\\
&=r^\alpha e^{\pi\alpha i/2}
\left(\cos{t}\cot{(\theta/2)}+\sin{t}\right)^{\alpha}
+\frac i2\cdot
\frac{n\alpha(1+\cot^{2}{(\theta/2)})\cos{t}}{\cos{t}\cot{(\theta/2)}+\sin{t}}.
\end{aligned}
\end{equation*}
Let $x=\cot{(\theta/2)}\cos{t}+\sin{t}.$
When $x>0,$ we write $R_{\alpha,c,n}(e^{i\theta})=u_+(x)+iv_+(x)$ and
get the expressions
\begin{equation*}
\left\{
\begin{aligned}
u_+(x)&=a(rx)^{\alpha},\\
v_+(x)&=b(rx)^{\alpha}+\frac{n\alpha}{2\cos{t}}\left(x-2\sin t+\frac1x\right),
\end{aligned}
\right.
\end{equation*}
where 
$$
a=\cos\frac{\beta\pi}{2}
\aand
b=\sin\frac{\beta\pi}{2}.
$$
Taking the derivative, we get
\begin{equation*}
v_+'(x)=\frac{n\alpha}{2x^2\cos{t}}
\left[x^2+\frac{2br^{\alpha}\cos{t}}{n}x^{\alpha+1}-1\right].
\end{equation*}
Hence, the minimum of $v_+(x)$ is attained at $x=\xi(A,\alpha),$
where $A=2br^\alpha n\inv\cos t.$
By using the relation \eqref{eq:eq}, we obtain
\begin{align*}
\min_{0<x}v_+(x)&=v_+(\xi)
=\frac{n}{2\cos t}\left(A\xi^\alpha+\alpha\xi+\frac{\alpha}{\xi}\right)
-n\alpha\tan t \\
&=\frac{n}{2\cos t}\left((\alpha-1)\xi+\frac{\alpha+1}{\xi}\right)-n\alpha\tan t
=U(\xi),
\end{align*}
where 
$$
U(x)=\frac{n}{2\cos t}\left((\alpha-1)x+\frac{\alpha+1}{x}\right)-n\alpha\tan t.
$$
Since the function $U(x)$ is decreasing in $0<x<1,$
Lemma \ref{lem:sol} yields the inequality
\begin{align*}
v_+(\xi)&=U(\xi)\ge
U((1+A)^{-1/2}) \\
&=\frac{n}{2\cos t}\left(\frac{\alpha-1}{\sqrt{1+A}}+(\alpha+1)\sqrt{1+A}\right)
-n\alpha\tan t.
\end{align*}
We remark here that
$$
U((1+A)^{-1/2})>U(1)=\frac{n\alpha(1-\sin t)}{\cos t}>0;
$$
namely, $v_+(x)>0$ for $x>0.$
When $x<0,$ letting $y=-x=-\cot{(\theta/2)}\cos{t}-\sin{t},$
we write $h(e^{i\theta})=u_-(y)+iv_-(y).$
Then, with the same $a$ and $b$ as above,
\begin{equation*}
\left\{
\begin{aligned}
u_-(y)&=a(ry)^{\alpha},\\
v_-(y)&=-b(ry)^{\alpha}-\frac{n\alpha}{2\cos{t}}\left(y+2\sin t+\frac1y\right),
\end{aligned}
\right.
\end{equation*}
We observe here that $u_+=u_->0$ and, in particular, we obtain the following.

\begin{lemma}\label{lem:1}
The left half-plane $\Omega_1=\{w: \Re w<0\}$
is contained in $R_{\alpha,c,n}(\mathbb{D}).$
\end{lemma}

We now look at $v_-(y).$
Since
$$
v_-'(y)=-\frac{n\alpha}{2y^2\cos{t}}
\left[y^2+\frac{2br^{\alpha}\cos{t}}{n}y^{\alpha+1}-1\right],
$$
in the same way as above, we obtain
\begin{align*}
\max_{0<y}v_-(y)&=v_-(\xi)
=-\frac{n}{2\cos t}\left((\alpha-1)\xi+\frac{\alpha+1}{\xi}\right)
-n\alpha\tan t \\
&\le-\frac{n}{2\cos t}\left(\frac{\alpha-1}{\sqrt{1+A}}+(\alpha+1)\sqrt{1+A}\right)
-n\alpha\tan t,
\end{align*}
where $\xi=\xi(A,\alpha)$ and $A=2br^\alpha n\inv\cos t.$
Note also that $v_-(y)<0$ for $y>0.$

Since the horizontal parallel strip $v_-(\xi)<\Im w< v_+(\xi)$ is contained
in the image domain $R_{\alpha,c,n}(\mathbb{D})$ of the generalized open door function,
we obtain the following.

\begin{lemma}\label{lem:2}
The parallel strip $\Omega_2$ described by
$$
\left|\Im w+n\alpha\tan t\right|
<\frac{n}{2\cos t}\left(\frac{\alpha-1}{\sqrt{1+A}}+(\alpha+1)\sqrt{1+A}\right)
$$
is contained in $R_{\alpha,c,n}(\mathbb{D}).$
Here, $t=\arg c\in(-\frac\pi2,\frac\pi2)$
and $A=\frac2n|c|^{\alpha}\sin\frac{\pi\alpha}2\cos t.$
\end{lemma}

When $\alpha=1,$ we have $u_\pm=0,$ that is, the boundary is contained in
the imaginary axis.
Since $\xi(A,1)=(1+A)^{-1/2}$ by Lemma \ref{lem:sol},
the above computations tell us
$\min v_+=(n/\cos t)(\sqrt{1+A}-\sin t)=C_n(\bar c).$
Similarly, we have $\max v_-=-(n/\cos t)(\sqrt{1+A}+\sin t)=-C_n(c).$
Therefore, we have $R_{1,c,n}(\mathbb{D})=V(-C_n(c), C_n(\bar c)).$
Note that the open door function then takes the following form
\begin{align*}
R_{1,c,n}(z)&=\frac{c+\bar{c}z}{1-z}
+\frac{2n(\Re c)z}{(1-z)(c+\bar{c}z)} \\
&=\frac{2\Re c+n}{1+cz/\bar c}-\frac n{1-z}-\bar c,
\end{align*}
which is the same as given by Kuroki and Owa \cite[(2.2)]{KO14}.
In this way, we see that Theorem \ref{thm:main} contains Theorem \ref{Thm:ODL}
as a special case.

\begin{remark}
In ~\cite{KO14}, they proposed another open door function of the form
\begin{equation*}
R(z)=\frac{2n|c|}{\Re c}\sqrt{\frac{2\Re c}{n}+1}
\frac{(\zeta-z)(1-\bar\zeta z)}{(1-\bar\zeta z)^2-(\zeta-z)^2}
-\frac{\Im c}{\Re c}i,
\end{equation*}
where
\begin{equation*}
\zeta=1-\frac{2}{\omega},\quad
\omega=\frac{c}{|c|}\sqrt{\frac{2\Re c}{n}+1}+1.
\end{equation*}
It can be checked that $R(z)=R_{1,c,n}(-\omega z/\bar\omega).$
Hence, $R$ is just a rotation of $R_{1,c,n}.$
\end{remark}

We next study the argument of the boundary curve of $R_{\alpha,c,n}(\mathbb{D}).$
We will assume that $0<\alpha<1$ since we have nothing to do when $\alpha=1.$

As we noted above, the boundary is contained in the right half-plane $\Re w>0.$
When $x>0$, we have
\begin{equation*}
\frac{v_+(x)}{u_+(x)}
=\frac{b}{a}+\frac{n\alpha}{2ar^\alpha x^\alpha\cos t}
\left[x+\frac{1}{x}-2\sin{t}\right].
\end{equation*}
We observe now that $v_+(x)/u_+(x)\to+\infty$ as $x\to0+$ or $x\to+\infty.$
We also have
\begin{equation*}
\left(\frac{v_+}{u_+}\right)'(x)
=\frac{n\alpha}{2ar^\alpha x^{\alpha+2}\cos{t}}
\left[(1-\alpha)x^{2}+2\alpha x\sin{t}-(1+\alpha)\right].
\end{equation*}
Therefore, $v_+(x)/u_+(x)$ takes its minimum at $x=\xi,$ where
$$
\xi=\frac{-\alpha\sin t+\sqrt{1-\alpha^2\cos^2t}}{1-\alpha}
$$
is the positive root of the equation $(1-\alpha)x^{2}+2\alpha x\sin{t}-(1+\alpha)=0.$
It is easy to see that $1<\xi$ and that
\begin{align*}
T_+&:=\min_{0<x}\frac{v_+(x)}{u_+(x)}
=\frac{v_+(\xi)}{u_+(\xi)}
=\frac{b}{a}+\frac{n\alpha}{2ar^\alpha \xi^\alpha\cos t}
\left[\xi+\frac{1}{\xi}-2\sin{t}\right] \\
&=\tan\frac{\pi\alpha}2+\frac{n(\xi-\xi\inv)}{2ar^\alpha \xi^\alpha\cos t}.
\end{align*}
When $x=-y<0,$ we have
\begin{equation*}
\frac{v_-(y)}{u_-(y)}
=-\frac{b}{a}-\frac{n\alpha}{2ar^\alpha y^\alpha\cos t}
\left[y+\frac{1}{y}+2\sin{t}\right]
\end{equation*}
and
\begin{equation*}
\left(\frac{v_-}{u_-}\right)'(y)
=\frac{-n\alpha}{2ar^\alpha y^{\alpha+2}\cos{t}}
\left[(1-\alpha)y^{2}-2\alpha y\sin{t}-(1+\alpha)\right].
\end{equation*}
Hence, $v_-(y)/u_-(y)$ takes its maximum at $y=\eta,$ where
$$
\eta=\frac{\alpha\sin t+\sqrt{1-\alpha^2\cos^2t}}{1-\alpha}.
$$
Note that
$$
T_-:=\max_{0<y}\frac{v_-(y)}{u_-(y)}
=\frac{v_-(\eta)}{u_-(\eta)}
=-\tan\frac{\pi\alpha}2-\frac{n(\eta-\eta\inv)}{2ar^\alpha \eta^\alpha\cos t}.
$$
Therefore, the sector $\{w: T_-<\arg w<T_+\}$ is contained in the image $h(\mathbb{D}).$
It is easy to check that $T_-<-\tan(\pi\alpha/2)<\tan(\pi\alpha/2)<T_+.$
In particular $T_-<\arg c^\alpha=\alpha t<T_+.$
We summarize the above observations, together with Theorem \ref{thm:main},
in the following form.

\begin{corollary}\label{cor:sector}
Let $0<\alpha< 1$ and $c=re^{it}$ with $r>0, -\pi/2<t<\pi/2,$ and $n$ be a
positive integer.
If a function $q\in\mathcal{H}[c^\alpha,n]$ satisfies the condition
$$
-\Theta_-<\arg\left(q(z)+\frac{zq'(z)}{q(z)}\right)<\Theta_+
$$
on $|z|<1,$ then $|\arg q|<\pi\alpha/2$ on $\mathbb{D}.$
Here,
$$
\Theta_\pm=\arctan\left[
\tan\frac{\pi\alpha}2+
\frac{n(\xi_\pm-\xi_\pm\inv)}{2r^\alpha \xi_\pm^\alpha\cos(\pi\alpha/2)\cos t}
\right],
$$
and
$$
\xi_\pm=\frac{\mp\alpha\sin t+\sqrt{1-\alpha^2\cos^2t}}{1-\alpha}.
$$
\end{corollary}

It is a simple task to check that $x^{1-\alpha}-x^{-1-\alpha}$ is increasing
in $0<x.$
When $\Im c>0,$ we see that $\xi_->\xi_+$ and thus $\Theta_->\Theta_+.$
It might be useful to note the estimates $\xi_-<\sqrt{(1+\alpha)/(1-\alpha)}<\xi_+$ and
$\xi_-<1/\sin t$ for $\Im c>0.$

\begin{remark}\label{rem:Mocanu}
When $c=1$ and $n=1,$ we have $\xi:=\xi_\pm=\sqrt{(1+\alpha)/(1-\alpha)},~
\xi-\xi\inv=2\alpha/\sqrt{1-\alpha^2},$ and thus
\begin{align*}
\Theta_\pm&=\arctan\left[
\tan\frac{\pi\alpha}2+
\frac{\xi-\xi\inv}{2\xi^\alpha\cos\tfrac{\pi\alpha}2}
\right] \\
&=\arctan\left[
\tan\frac{\pi\alpha}2+
\frac{\alpha}{\cos\tfrac{\pi\alpha}2(1-\alpha)^{\frac{1-\alpha}2}(1+\alpha)^{\frac{1+\alpha}2}}
\right] \\
&=\frac{\pi\alpha}2+
\arctan\left[
\frac{\alpha\cos\tfrac{\pi\alpha}2}%
{(1-\alpha)^{\tfrac{1-\alpha}2}(1+\alpha)^{\tfrac{1+\alpha}2}+\alpha\sin\tfrac{\pi\alpha}2}
\right].
\end{align*}
Therefore, the corollary gives a theorem proved by Mocanu \cite{Mocanu89}.
\end{remark}

Since the values $\Theta_+$ and $\Theta_-$ are not given in an explicitly way,
it might be convenient to have a simpler sufficient condition for
$|\arg q|<\pi\alpha/2.$

\begin{corollary}
Let $0<\alpha\le1$ and $c$ with $\Re c>0$ and $n$ be a
positive integer.
If a function $q\in\mathcal{H}[c^\alpha,n]$ satisfies the condition
$$
q(z)+\frac{zq'(z)}{q(z)}\in\Omega,
$$
then $|\arg q|<\pi\alpha/2$ on $\mathbb{D}.$
Here, $\Omega=\Omega_1\cup\Omega_2\cup\Omega_3,$ and $\Omega_1$ and $\Omega_2$
are given in Lemmas \ref{lem:1} and \ref{lem:2}, respectively, and
$\Omega_3=\{w\in\mathbb{C}: |\arg w|<\pi\alpha/2\}.$
\end{corollary}

\begin{proof}
Lemmas \ref{lem:1} and \ref{lem:2} yield that $\Omega_1\cup\Omega_2\subset
R_{\alpha,c,n}(\mathbb{D}).$
Since $\Theta_\pm>\pi\alpha/2,$
we also have $\Omega_3\subset R_{\alpha,c,n}(\mathbb{D}).$
Thus $\Omega\subset R_{\alpha,c,n}(\mathbb{D}).$
Now the result follows from Theorem \ref{thm:main}.
\end{proof}

See Figure 1 for the shape of the domain $\Omega$ together with
$R_{\alpha,c,n}(\mathbb{D}).$
We remark that $\Omega=R_{\alpha,c,n}(\mathbb{D})$ when $\alpha=1.$

\begin{figure}[!bht]
\begin{center}
\scalebox{0.5}[0.46]{\includegraphics{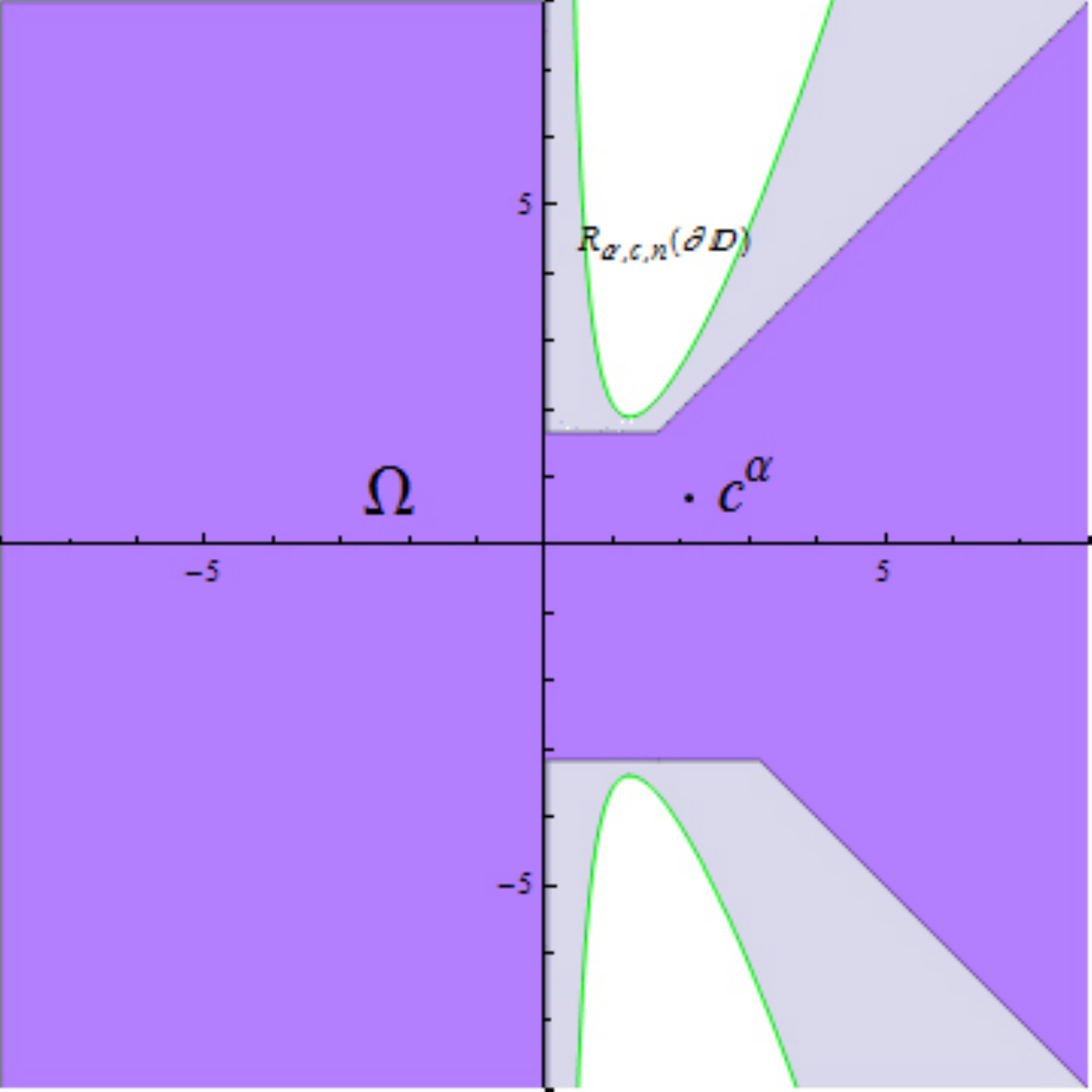}}
\caption{The image $R_{\alpha,c,n}(\mathbb{D})$ and $\Omega$ for $\alpha=1/2, c=4+3i, n=2.$}
\label{fig1}
\end{center}
\end{figure}

\def\cprime{$'$} \def\cprime{$'$} \def\cprime{$'$}
\providecommand{\bysame}{\leavevmode\hbox to3em{\hrulefill}\thinspace}
\providecommand{\MR}{\relax\ifhmode\unskip\space\fi MR }
\providecommand{\MRhref}[2]{%
  \href{http://www.ams.org/mathscinet-getitem?mr=#1}{#2}
}
\providecommand{\href}[2]{#2}

\end{document}